\documentclass[12pt]{article}

\usepackage{amsmath}
\usepackage{amsfonts}
\usepackage{amsthm}
\usepackage[english]{babel}
\usepackage{amssymb}

\newtheorem{theorem}{Theorem}
\newtheorem{lemma}{Lemma}

\newtheorem{proposition}{Proposition}
\newtheorem{remark}{Remark}

\newcommand{\mR}{\mathbb{R}}
\newcommand{\mC}{\mathbb{C}}
\newcommand{\mN}{\mathbb{N}}
\newcommand{\mE}{\mathbb{E}}

\newcommand{\mS}{\mathbb{S}}

\newcommand{\cH}{\mathcal{H}}

\newcommand{\cF}{\mathcal{F}}
\newcommand{\cP}{\mathcal{P}}

\newcommand{\cJ}{\mathcal{J}}

 \def\a{{\alpha}}

 \def\l{{\lambda}}
 
 \def\o{{\omega}}
 
 \def\la{{\langle}}
 \def\ra{{\rangle}}

\begin{document}

\title{The kernel of the radially deformed Fourier transform}

\author{H. De Bie\thanks{ E-mail: {\tt Hendrik.DeBie@ugent.be}}}

\date{\small{Clifford Research Group}\\
\small{Department of Mathematical Analysis\\
Faculty of Engineering and Architecture -- Ghent University\\ Krijgslaan 281, 9000 Gent,
Belgium}}

\maketitle

\begin{abstract}

The radially deformed Fourier transform, introduced in [S. Ben Sa\"id, T. Kobayashi and B. {\O}rsted, Laguerre semigroup and Dunkl operators, Compositio Math.], is an integral transform that depends on a numerical parameter $a \in \mR^{+}$. So far, only for $a=1$ and $a=2$ the kernel of this integral transform is determined explicitly. 

In the present paper, explicit formulas for the kernel of this transform are obtained when the dimension is even and $a = 2/n$ with $n \in \mN$. As a consequence, it is shown that the integral kernel is bounded in dimension 2.
\end{abstract}


\section{Introduction}
\setcounter{equation}{0}
\label{intro}

Harmonic analysis in $\mR^{m}$ is governed by the following three operators
\[
\Delta := \sum_{i=1}^{m}\partial_{x_{i}}^{2},\qquad
|x|^{2} :=  \sum_{i=1}^{m}x_{i}^{2},\qquad
\mE := \sum_{i=1}^{m} x_{i} \partial_{x_{i}}
\]
with $\Delta$ the Laplace operator and $\mE$ the Euler operator. As observed in \cite{H, HT}, the operators $E = |x|^{2}/2$, $F =-\Delta/2$ and $H =\mE + m/2$ are invariant under $O(m)$ and generate the Lie algebra $\mathfrak{sl}_{2}$:
\[
\big[H,E\big] = 2E,\qquad \big[H,F\big] = -2F,\qquad \big[E,F\big] = H.
\]

Recently, there has been a lot of interest in other differential or difference operator realizations of $\mathfrak{sl}_2$ or other Lie 
(super)algebras. The focus is in particular on the generalized Fourier transforms that subsequently arise. We mention the Dunkl transform \cite{deJ}, various discrete Fourier transforms \cite{AW, J}, Fourier transforms in Clifford analysis \cite{MR2190678, H12, DBXu}, etc. For a more detailed review, we refer the reader to \cite{DBR}.

A hard problem in this context is to find explicit closed formulas for the integral kernel of the associated Fourier transforms. This paper is concerned with a partial solution of this problem for one of the most important new realizations of this type.

 The set up is as follows. It can be observed that the $\mathfrak{sl}_{2}$ relations also hold for the generalized operators $|x|^a$, $|x|^{2-a} \Delta$ and $\mE + \frac{a+m-2}{2}$, with $a >0$ a real parameter. One then has the following commutators
\begin{eqnarray*}
\left[|x|^{2-a}\Delta, |x|^a \right] &=& 2 a \, (\mE + \frac{a +m -2}{2})\\
\left[|x|^{2-a}\Delta, \mE + \frac{a+m-2}{2} \right] &=& a\, |x|^{2-a}\Delta\\
\left[|x|^a, \mE + \frac{a +m-2}{2}\right] &=& -a \,|x|^a.
\end{eqnarray*}
This was first observed, in the context of minimal representations, for $a=1$ in \cite{MR2134314, MR2401813} and subsequently generalized to arbitrary $a$ in \cite{Orsted2}.

\begin{remark}
In fact, the paper \cite{Orsted2} studied an even more general deformation, where the Laplace operator was replaced by the Dunkl Laplace operator (see e.g. \cite{MR1827871}). As in that case, there is no hope to find explicit closed formulas for the $a$-deformation, we restrict ourselves in this paper to the ordinary Laplace operator.
\end{remark}

The paper \cite{Orsted2} was mostly concerned with the study of the associated Hermite semigroup given by
\[
\cJ_{a}(\o):=e^{\frac{\o}{a} \left( |x|^{2-a}\Delta - |x|^a\right)}
\]
where $\o$ is a complex parameter satisfying $\Re{\o} \geq 0$. This semigroup was studied in great detail and in particular an integral operator expression was found where the kernel is given as a series expansion. An important tool was the construction of an eigenbasis for the hamiltonian $H_{a} =- \left( |x|^{2-a}\Delta - |x|^a\right)/a$. Putting, for $j, k \in \mN$ and $H_k^{(\ell)} $, $(\ell = 1, \ldots, \dim{\cH_k })$ a basis for $\cH_k:= (\ker{\Delta}) \cap \cP_{k}$,
\begin{equation}
\label{HermRad}
\phi_{j, k, \ell}^{a} := L_{j}^{\frac{m + 2k-2}{a}}\left(\frac{2}{a}|x|^a\right) H_k^{(\ell)} \, e^{-|x|^a/a},
\end{equation}
lengthy computations show
\[
H_{a} \phi_{j, k, \ell}^{a} = \left( \frac{m - 2}{a} + \frac{2k}{a}+2j+1\right) \phi_{j, k, \ell}^{a}.
\]
The set of functions $\{ \phi_{j, k, \ell}^{a}\}$ forms an orthogonal basis for the space $L_2(\mR^m,  |x|^{a-2}dx)$.

In order to keep the resulting formulas as simple as possible, we restrict ourselves from here on to the specific semigroup parameter $\o= i \pi /2$ (although our results easily extend to arbitrary $\omega$). This yields the so-called radially deformed Fourier transform
\[
\cF_{a} = e^{ \frac{i \pi (m+a-2)}{2a}} e^{\frac{i \pi}{2 a}(|x|^{2-a}\Delta - |x|^{a})},
\]
where a suitable normalization has been added to make the transform unitary. A series expansion of its integral kernel is given in the subsequent theorem which was obtained in \cite{Orsted2}.

\begin{theorem}
\label{thm_exp}
Put
\[
K_{a}^{m}(x,y) = a^{2 \lambda/a} \Gamma \left( \frac{2 \lambda+a}{a}\right)\sum_{k=0}^{\infty} e^{-\frac{i \pi k}{a}}  \frac{\lambda +k}{\lambda} z^{-\lambda} J_{\frac{2(k+ \lambda)}{a}}\left( \frac{2}{a} z^{a/2}\right) \; C_{k}^{\lambda}(w),
\]
with $\lambda =(m-2)/2$, $z = |x| |y|$ and $w = \la x, y \ra /z$. This series converges absolutely and uniformly on compact subsets and the integral transform 
\[
\cF_{a} (f)(y) = \frac{\Gamma(m/2)}{ \Gamma (\frac{2 \lambda+a}{a}) 2 a^{2\lambda/a} \pi^{m/2}} \int_{\mR^{m}} K_{a}(x,y) f(x)  |x|^{a-2}dx
\]
defined on the function space  $L_2(\mR^m, |x|^{a-2}dx)$  coincides with the operator\\ $\cF_{a} = e^{ \frac{i \pi (m+a-2)}{2a}} e^{\frac{i \pi}{2 a}(|x|^{2-a}\Delta - |x|^{a})}$ on the basis $\phi_{j, k, \ell}^{a}$:
\begin{equation}
\label{eigvals}
\cF_{a}  \left(\phi_{j, k, \ell}^{a}  \right)= e^{-i \pi (j + \frac{k}{a})}\phi_{j, k, \ell}^{a}.
\end{equation}
\end{theorem}

Formally, this theorem can be obtained by combining the integral identity (see \cite[exercise 21, p. 371]{Sz})
\[
\int_{0}^{+\infty} r^{\alpha+1}  J_{\a}(rs)\,  L_{j}^{\a}(r^{2}) e^{-r^{2}/2}dr = (-1)^{j}s^{\a} L_{j}^{\a}(s^{2}) e^{-s^{2}/2}
\]
with the fact that the Gegenbauer polynomial $C_k^\lambda$ with $\l= (m-2)/2$ yields the reproducing kernel for the space of 
spherical harmonics of degree $k$ (cf. \cite{V}). This means that for $\xi, \eta \in \mS^{m-1}$
\[
\frac{\lambda+k}{\l} \int_{\mS^{m-1}} C^{\lambda}_{k}(\langle \xi,\eta \rangle) H_{\ell}(\xi) d \sigma(\xi) = \sigma_m \, \delta_{k \ell} \, H_{\ell}(\eta), \qquad H_{\ell} \in \cH_{\ell}
\]
with $\sigma_m = 2 \pi^{m/2}/\Gamma(m/2)$. The absolute and uniform convergence of the series is established in Lemma 4.17, \cite{Orsted2}.

Note that for $a=2$, the kernel $K_{a}(x,y)$ reduces to the usual exponential kernel of the ordinary Fourier transform:
\begin{align*}
K_{2}^{m}(x,y) &= 2^{\lambda} \Gamma(\lambda)\sum_{k=0}^{\infty}(-i)^{k}(k+ \lambda)  z^{-\lambda} J_{k+ \lambda}(z) \; C_{k}^{\lambda}(w)\\
&= e^{-i\la x, y \ra},
\end{align*}
see \cite[Section 11.5]{W}).

Also when $a=1$, a closed form is known, given by
\begin{equation}
\label{kernela1}
K_{1}^{m}(x,y) = \Gamma \left(\frac{m-1}{2}\right)  \widetilde{J}_{\frac{m-3}{2}} \left( \sqrt{2 (|x| |y| + \la x, y \ra )} \right),
\end{equation}
with $\widetilde{J}_{\nu}(z) = (z/2)^{-\nu}J_{\nu}(z)$. This result was announced in \cite{MR2134314} and proven in \cite{MR2401813} using a rather cumbersome geometric construction.

For arbitrary $a$ such a closed form is not available. Moreover, there are no bounds known on $K_{a}(x,y)$ for $a \neq 1$ or $\neq2$. Also a characterization of the kernel $K_{a}(x,y)$ as the unique eigenfunction of a system of PDEs is not known.

The strategy we will follow in this paper to determine an explicit formula for the series expansion in Theorem \ref{thm_exp} depends on two essential steps:
\begin{itemize}
\item find a recursion property on the dimension
\item use a trick to find the explicit formula in dimension 2.
\end{itemize}
Using these two steps, we are able to find an explicit formula for the kernel of the radially deformed Fourier transform for $a=2/n$ with $n \in \mN$.

\begin{remark}
After this paper was finished, the author was pointed to results of N. Demni in \cite{ND}. There, a closed formula is obtained for the generalized Bessel function (in the Dunkl sense) related to dihedral groups. In his computation, similar formulas appear as in the present paper, in particular in Lemma 2.1 and Corollary 2.2.
\end{remark}

The paper is organized as follows. In Section \ref{red} we describe the method to reduce the dimension. In Section \ref{secdim2} we determine the kernel explicitly in dimension 2, whenever $a= 2/n$. Finally, in Section \ref{kerndimeven} we find explicit formulas for the kernel in all even dimensions.

\section{Reduction of dimension}
\setcounter{equation}{0}
\label{red}

Put
\begin{equation}
\label{KernelAFT}
K_{a}^{m}(z,w) = a^{2 \lambda/a} \Gamma \left( \frac{2 \lambda+a}{a}\right)\sum_{k=0}^{\infty} e^{-\frac{i \pi k}{a}}  \frac{\lambda +k}{\lambda} z^{-\lambda} J_{\frac{2(k+ \lambda)}{a}}\left( \frac{2}{a} z^{a/2}\right) \; C_{k}^{\lambda}(w)
\end{equation}
with $\lambda =(m-2)/2$.

In the following lemma, we find the relation between $K_{a}^{m}$ and $K_{a}^{m+2}$.

\begin{lemma}
\label{reckernel}
One has
\[
K_{a}^{m+2}(z,w) = e^{i\frac{\pi}{a}} a^{\frac{2}{a}} \frac{\Gamma \left( \frac{2 \l + a+2}{a}\right)}{2 (\l+1) \Gamma \left( \frac{2 \l + a}{a}\right)}z^{-1} \partial_w K_{a}^{m}(z,w)
\]
with $\lambda =(m-2)/2$.
\end{lemma}

\begin{proof}
This follows immediately by applying the derivation formula for the Gegenbauer polynomials, see e.g.  \cite[(4.7.14)]{Sz} ,
\[
\frac{d}{d w} C_k^{\l}(w) = 2 \l C_{k-1}^{\l+1}(w)
\]
to formula (\ref{KernelAFT}). Derivation and summation can be interchanged, because both $K_{a}^{m}$ and $K_{a}^{m+2}$ converge absolutely and uniformly on compact subsets of $(z,w) \in \mR^+ \times [-1,1]$, see Theorem 1.
\end{proof}

This lemma allows us to reduce the problem of determining the kernel to dimension 2, resp. 3. Precise formulas are given in the following proposition.

\begin{proposition}
\label{recursionresult}
When the dimension $m = 2 n$ is even, the kernel of the radially deformed FT is given by
\[
K_{a}^{2n}(z,w) = \frac{e^{i (n-1)\frac{\pi}{a}} a^{\frac{2(n-1)}{a}}}{(2n-2)!!} \Gamma \left( \frac{2n-2+a}{a}\right)z^{-n+1} \partial_w^{n-1} K_a^{2}(z,w).
\]

When the dimension $m = 2 n+1$ is odd, the kernel of the radially deformed FT is given by
\[
K_{a}^{2n+1}(z,w) =  \frac{e^{i (n-1)\frac{\pi}{a}} a^{\frac{2(n-1)}{a}}}{(2n-1)!!} \frac{\Gamma \left( \frac{2n-1+a}{a}\right)}{\Gamma \left( \frac{1+a}{a}\right)}z^{-n+1} \partial_w^{n-1} K_a^{3}(z,w).
\]
\end{proposition}

\begin{proof}
Immediately, by repeated application of Lemma \ref{reckernel}.
\end{proof}

\begin{remark}
Note that, in the odd-dimensional case, we cannot reduce the problem to determining the kernel for $m=1$. Indeed, in that case, the expansion in formula (\ref{KernelAFT}) is no longer valid.
\end{remark}

\section{The case of dimension 2}
\setcounter{equation}{0}
\label{secdim2}

The kernel in dimension $m=2$ is given by the expression for $K_{a}^{m}(z,w)$ with $\l =0$. We obtain
\begin{align*}
K_{a}^{2}(z,w) &= \lim_{\l \rightarrow 0} K_{a}^{m}(z,w)\\
& = J_0 \left(\frac{2}{a} z^{a/2}  \right) +  \sum_{k=1}^{\infty} k \, e^{- \frac{i\pi k}{a}}  J_{\frac{2k}{a}} \left(\frac{2}{a} z^{a/2}  \right)\lim_{\lambda \rightarrow 0} \l^{-1} C_k^{\l}(w)\\
&=J_0 \left(\frac{2}{a} z^{a/2}  \right) + 2 \sum_{k=1}^{\infty} e^{- \frac{i\pi k}{a}}  J_{\frac{2k}{a}} \left(\frac{2}{a} z^{a/2}  \right) \cos{ k \theta}
\end{align*}
with $w = \cos{\theta}$.
Here we used the well-known relation \cite[(4.7.8)]{Sz} 
$$
\lim_{\lambda \rightarrow 0}  \lambda^{-1} C_k^\lambda (w) = (2/k) \cos k \theta, \quad w = \cos \theta, \quad k \geq 1.   
$$

We now find a closed formula for
\[
f_{a}(z,w)=J_0 \left(z \right) + 2 \sum_{k=1}^{\infty} e^{- \frac{i\pi k}{a}}  J_{\frac{2k}{a}} \left(z \right) \cos{ k \theta}, \qquad w = \cos \theta,
\]
whenever $a=2/n$ with $n \in \mN$. We start by proving an important auxiliary result.
\begin{lemma}
\label{auxlemma}
Let 
\[
f(t) = \sum_{k=0}^{+\infty} a_{k} \cos{k t}, \qquad a_{k} \in \mC
\]
be an absolutely convergent Fourier series. Then the series
\[
g(t) = \sum_{k=0}^{+\infty} a_{n k} \cos{k t}
\]
is given explicitly by
\[
g(t) = \frac{1}{n} \sum_{j=0}^{n-1} f \left( \frac{t + 2 \pi j}{n} \right).
\]
\end{lemma}

\begin{proof}
By decomposing the formula to be proven into real and imaginary parts, we observe that it suffices to prove the statement for $a_{k}$ real. Under this assumption, consider the function
\[
F(t) = \sum_{k=0}^{+\infty} a_{k} e^{ i k t},
\]
which satisfies $\Re{F(t)} = f(t)$.
This function can be decomposed as
\[
F(t) = \sum_{j=0}^{n-1}  F_{j}(t)
\]
with
\[
F_{j} (t) = \sum_{k=0}^{\infty} a_{n k +j} e^{i (nk+j) t}.
\]
Observe that $F_{j}(t + \frac{2 \pi \ell}{n})= e^{i \frac{2 \pi j \ell}{n}} F_{j}(t)$. Next we calculate
\begin{align*} 
\frac{1}{n} \sum_{\ell=0}^{n-1} F\left( \frac{t + 2 \pi \ell}{n} \right) &= \frac{1}{n} \sum_{\ell=0}^{n-1}  \sum_{j=0}^{n-1}  F_{j}\left( \frac{t + 2 \pi \ell}{n} \right)\\
&=\frac{1}{n}  \sum_{j=0}^{n-1} \sum_{\ell=0}^{n-1}  e^{i \frac{2 \pi j \ell}{n}}  F_{j}\left( \frac{t }{n} \right)\\
&= F_{0}\left( \frac{t }{n} \right).
\end{align*}
In the last step, we used the fact that
\[
\sum_{\ell=0}^{n-1}  e^{i \frac{2 \pi j \ell}{n}} = \left\{ \begin{array}{l} n, \quad j=0\\0, \quad j \in \{1, \ldots, n-1\}  \end{array} \right.
\]

The proof is now completed by observing that $g(t) = \Re F_{0}\left( \frac{t }{n} \right)$.
\end{proof}

We can now give a closed formula for the kernel when the parameter $a =2/n$, with $n \in \mN$.
\begin{theorem}
\label{ClosedForm}
A closed formula for the series
\[
f_{\frac{2}{n}}(z,t)=J_0 \left(z \right) + 2 \sum_{k=1}^{\infty} e^{- \frac{i\pi k n}{2}}  J_{ k n} \left(z \right) \cos{ k t}
\]
is given by
\[
f_{\frac{2}{n}}(z,t) = \frac{1}{n} \sum_{\ell=0}^{n-1}e^{-i z \cos{\frac{t+2 \pi \ell}{n}}}.
\]
\end{theorem}

\begin{proof}
The following expansion is well-known:
\begin{equation}
\label{FTdecomp}
e^{-i z \cos{t}} = J_{0}(z) +2 \sum_{k=1}^{\infty} (-i)^{k}  J_{k} \left(z \right) \cos{ k t},
\end{equation}
see e.g. \cite{W}, Section 11.5.  Application of Lemma \ref{auxlemma} to formula (\ref{FTdecomp}) then yields the result.
\end{proof}

In the following theorem, we give the kernel of the radially deformed Fourier transform in dimension two when $a =2/n$ and obtain its boundedness.
\begin{theorem}
Let $a =2/n$, with $n \in \mN$. Then the kernel
\[
K_{a}^{2}(z,w) =J_0 \left(\frac{2}{a} z^{a/2}  \right) + 2 \sum_{k=1}^{\infty} e^{- \frac{i\pi k}{a}}  J_{\frac{2k}{a}} \left(\frac{2}{a} z^{a/2}  \right) \cos{ k \theta}
\]
with $w = \cos{\theta}$
is given in closed form by
\[
K_{\frac{2}{n}}^{2}(z,w) = \frac{1}{n} \sum_{\ell=0}^{n-1}e^{-i nz^{\frac{1}{n}} \cos{\frac{t}{n}}\cos{\frac{2 \pi \ell}{n}}}e^{i nz^{\frac{1}{n}} \sin{\frac{t}{n}}\sin{\frac{2 \pi \ell}{n}}}, \qquad t= \arccos{w}
\]
and satisfies the bound
\[
| K_{\frac{2}{n}}^{2}(z,w)| \leq 1.
\]
\end{theorem}

\begin{proof}
Using Theorem \ref{ClosedForm}, we obtain, with $t= \arccos{w}$,
\begin{align*}
K_{\frac{2}{n}}^{2}(z,w) &= \frac{1}{n} \sum_{\ell=0}^{n-1}e^{-i nz^{\frac{1}{n}} \cos{\frac{t+2 \pi \ell}{n}}}\\
&= \frac{1}{n} \sum_{\ell=0}^{n-1}e^{-i nz^{\frac{1}{n}} \cos{\frac{t}{n}}\cos{\frac{2 \pi \ell}{n}}}e^{i nz^{\frac{1}{n}} \sin{\frac{t}{n}}\sin{\frac{2 \pi \ell}{n}}}.
\end{align*}
The bound given in the statement of the theorem now follows immediately.
\end{proof}

We can simplify these formulas even further, depending on the parity of $n$. When $n=2 k$ is even, we obtain, using basic trigonometric identities
\[
K_{\frac{1}{k}}^{2}(z,w) = \frac{1}{k} \sum_{\ell=0}^{k-1} \cos{\left( nz^{\frac{1}{n}} \cos{\frac{t}{n}}\cos{\frac{ \pi \ell}{k}}\right)}\cos{\left(nz^{\frac{1}{n}} \sin{\frac{t}{n}}\sin{\frac{\pi \ell}{k}}\right)}.
\]
In other words, the kernel is a real-valued function. This was to be expected from formula (\ref{eigvals}), as the eigenvalues of the radially deformed FT are all real in this case. When $k=1$, the formula for the kernel reduces to
\begin{align*}
K_{1}^{2}(z,w) &= \cos{\left( 2 z^{\frac{1}{2}} \cos{\frac{t}{2}}\right)}\\
&=  \cos{\left( \sqrt{2 (z+z w )}\right)},
\end{align*}
which is in correspondence with formula (\ref{kernela1}) for $m=2$. Note that such expressions do not exist for $k>1$, as this would require simple algebraic formulas expressing $\cos{(t/2k)}$ in terms of $\cos{t}$.

When $n=2 k+1$ is odd, we obtain in a similar way
\begin{align*}
K_{\frac{2}{2k+1}}^{2}(z,w) &=\frac{1}{2k+1}\cos{\left( nz^{\frac{1}{n}} \cos{\frac{t}{n}}\right)} \\
&  + \frac{2}{2k+1} \sum_{\ell=1}^{k} \cos{\left( nz^{\frac{1}{n}} \cos{\frac{t}{n}}\cos{\frac{2 \pi \ell}{2k+1}}\right)}\cos{\left(nz^{\frac{1}{n}} \sin{\frac{t}{n}}\sin{\frac{2\pi \ell}{2k+1}}\right)}\\
&- \frac{i}{2k+1}  \sin{\left( nz^{\frac{1}{n}} \cos{\frac{t}{n}}\right)} \\
&- \frac{2i}{2k+1} \sum_{\ell=1}^{k} \sin{\left( nz^{\frac{1}{n}} \cos{\frac{t}{n}}\cos{\frac{2 \pi \ell}{2k+1}}\right)}\cos{\left(nz^{\frac{1}{n}} \sin{\frac{t}{n}}\sin{\frac{2\pi \ell}{2k+1}}\right)}.
\end{align*}

\section{The kernel in even dimension}
\setcounter{equation}{0}
\label{kerndimeven}

By combining 
\[
K_{\frac{2}{n}}^{2}(z,w) = \frac{1}{n} \sum_{\ell=0}^{n-1}e^{-i nz^{\frac{1}{n}} \cos{\frac{t+2 \pi \ell}{n}}}
\]
with Proposition \ref{recursionresult}, we obtain the kernel in all even dimensions.
\begin{proposition}
When the dimension $m = 2 k$ is even and $a= 2/n$, the kernel of the radially deformed FT is given by
\[
K_{\frac{2}{n}}^{2k}(z,w) =\left(\frac{2i}{n}\right)^{n(k-1)} \frac{ \left( nk-n\right)!}{(2k-2)!!}z^{-k+1} \partial_w^{k-1}  \left(\frac{1}{n} \sum_{\ell=0}^{n-1}e^{-i nz^{\frac{1}{n}} \cos{\frac{t+2 \pi \ell}{n}}} \right)
\]
with $t= \arccos{w}$.
\end{proposition}

Note that the iterated derivative in this proposition can be computed explicitly. However, the result is not very nice and does not immediately yield a bound on the kernel for all even dimensions.
Of course, when $n=2$, one reobtains the result of formula (\ref{kernela1}) as can easily be checked.

\end{document}